\documentclass[11pt]{amsart}

\usepackage{amsmath}
\usepackage{amsfonts}
\usepackage{amssymb}
\usepackage{amsthm}
\usepackage{epsfig}
\usepackage{graphicx}
\usepackage{url, hyperref}
\usepackage{nicefrac}

\newtheorem{theorem}{Theorem}[section]
\newtheorem{lemma}{Lemma}[section]
\newtheorem{corollary}{Corollary}[section]
\newtheorem{definition}{Definition}[section]
\newtheorem{remark}{Remark}[section]
\newtheorem{proposition}{Proposition}[section]
\newtheorem{conjecture}{Conjecture}[section]

\DeclareMathOperator{\inter}{int}

\def\conv{\mathop\mathrm{conv}\nolimits}

\def\K{\mathcal{K}}
\def\P{\mathcal{P}}
\def\L{\mathcal{L}}
\def\R{\mathbb{R}}

\def\Z{\mathbb{Z}}
\def\N{\mathbb{N}}

\newcommand{\ind}{\mathrm{ind}}
\newcommand{\vol}{\mathrm{vol}}
\newcommand{\LE}{\mathrm{G}}
\newcommand{\lE}{\mathrm{g}}
\newcommand{\LL}{\mathrm{L}}

\newcommand{\aff}{\mathrm{aff}}
\newcommand{\lin}{\mathrm{lin}}
\newcommand{\fac}{\mathcal{F}}
\newcommand{\set}{\mathcal{X}}

\numberwithin{equation}{section}

\begin{document}

\title[Lattice points of zonotopes and lattice-face polytopes]{Notes on lattice points of\\ zonotopes and lattice-face polytopes}
\author{Christian Bey, Martin Henk, Matthias Henze and Eva Linke}

\address{Fachbereich 3, Mathematik, Universit\"at Bremen, 
Postfach 330 440, 28334 Bremen, Germany}
\email{bey@math.uni-bremen.de}
\address{Fakult\"at f\"ur Mathematik, Otto-von-Guericke
Universit\"at Mag\-deburg, Universit\"atsplatz 2, D-39106 Magdeburg,
Germany}
\email{\{martin.henk,matthias.henze,eva.linke\}@ovgu.de}

\thanks{The work of the last two authors was supported by the Deutsche
  Forschungsgemeinschaft (DFG) within the project He 2272/4-1.}



\begin{abstract}
Minkowski's second theorem on successive minima gives an upper bound
on the volume of a convex body in terms of its successive minima. We
study the problem to generalize Minkowski's bound by replacing the volume by the lattice point
enumerator of a convex body. In this context we are 
interested in bounds on the coefficients of  Ehrhart
polynomials of  lattice polytopes via the successive minima. Our results for
lattice zonotopes and lattice-face polytopes imply, in particular,  that for  $0$-symmetric
lattice-face polytopes and lattice parallelepipeds the volume can be replaced
by the lattice point enumerator.
\end{abstract}

\maketitle

\section{Introduction}

Let $\K^n$ be the set of all convex bodies in $\R^n$, i.e., compact convex sets with non-empty interior. The additional subscript in $\K^n_0$ points out that the considered convex bodies are $0$-symmetric. When dealing with polytopes we write $\P^n$ and $\P^n_0$, and for $P\in\P^n$ we denote by $\mathrm{vert}(P)$ its set of vertices. The family of $n$-dimensional lattices in $\R^n$ is written as $\L^n$ and the usual Lebesgue measure with respect to the $n$-dimensional space as $\vol_n(\cdot)$. If the ambient space is clear from the context we omit the subscript and just write $\vol(\cdot)$. For some subset $K\subset\R^n$ and some lattice $\Lambda\in\L^n$ the lattice point enumerator is denoted by $\LE(K,\Lambda)=\#(K\cap\Lambda)$. If $\Lambda=\Z^n$ we shortly write $\LE(K)=\LE(K,\Z^n)$. In the following we study relations between this quantity and {\em Minkowski's successive minima} which are defined as
\begin{equation*}
\lambda_i(K,\Lambda)=\min\{\lambda>0:\dim(\lambda K\cap\Lambda)\geq i\},\,1\leq i\leq n,
\end{equation*}
for a $0$-symmetric convex body $K\in\K^n_0$ with respect to a lattice $\Lambda\in\L^n$. Note that $\dim(S)$ denotes the dimension of the affine hull of $S\subset\R^n$. If $\Lambda=\Z^n$ we just write $\lambda_i(K)=\lambda_i(K,\Z^n)$. These numbers form an increasing sequence, so $\lambda_1(K,\Lambda)\leq\ldots\leq\lambda_n(K,\Lambda)$, and as functionals on $\K^n_0\times\L^n$ they are homogeneous of degree $-1$ in the first and of degree $1$ in the second argument. An important and deep result in the geometry of numbers is the following theorem which is usually referred to as Minkowski's second theorem on convex bodies (cf. \cite[pp.~376]{Gruber2007}).

\begin{theorem}[Minkowski, 1896]\label{thm_mink}
Let $K\in\K^n_0$ and $\Lambda\in\L^n$ be a lattice. Then,
\begin{equation*}
 \lambda_1(K,\Lambda)\cdot\ldots\cdot\lambda_n(K,\Lambda)\vol(K)\leq 2^n\det(\Lambda).
\end{equation*}
\end{theorem}

The relevance of this result is also illustrated by the big number of proofs and generalizations from various contexts (see \cite{HenkWills2008} for a survey report). A discrete version of Minkowski's theorem was proposed, and proved in the planar case, in \cite{BetkeHenkWills1993} where the volume is replaced by the lattice point enumerator of $K\in\K^n_0$.

\begin{conjecture}[Betke, Henk, Wills, 1993]\label{conj_main}
Let $K\in\K^n_0$ and $\Lambda\in\L^n$ be a lattice. Then,
\begin{equation*}
 \LE(K,\Lambda)\leq\prod_{i=1}^n\left\lfloor\frac2{\lambda_i(K,\Lambda)}+1\right\rfloor.
%\label{eq:conj_main}
\end{equation*}
\end{conjecture}

This conjecture would not only generalize Theorem \ref{thm_mink}\, but also unify this and other particular results from geometry of numbers, for example $\LE(K)\leq 3^n$, for $K\in\K^n_0$ whose only interior lattice point is the origin (cf. \cite[p.~79]{Minkowski1910}).
Recently, Malikiosis \cite{Malikiosis2010, Malikiosis2009} settled the
three-dimensional case by an inductive approach and obtained the
smallest known constant $c=\sqrt[3]{40/9}\approx1.64414$ such
that, roughly speaking, the conjecture holds up to the factor
$c^n$. Already proposed in \cite[Ch.~2, \textsection
9]{GruberLekkerkerker1987}, it is natural to extend the notion of
successive minima to general, not necessarily $0$-symmetric, convex
bodies $K\in\K^n$ via some symmetrization, e.g., by considering
$\lambda_i(\frac{1}{2}\,DK,\Lambda)$, where  $DK=K-K$. With this notation the above conjecture for $K\in\K^n$ reads
\begin{equation}
 \LE(K,\Lambda)\leq\prod_{i=1}^n\left\lfloor\frac1{\lambda_i(DK,\Lambda)}+1\right\rfloor,
\label{eq:conj_genP}
\end{equation}
and we will mostly deal with this more general question.\par

A helpful observation is, that it suffices to prove
\eqref{eq:conj_genP} for lattice polytopes $P\in\P^n$, i.e.,
$\mathrm{vert}(P)\subset\Lambda$. Indeed, since the successive minima
are monotonic functionals, i.e., if $K,K'\in\K^n_0$ with $K\subseteq
K'$, then $\lambda_i(K,\Lambda)\geq\lambda_i(K',\Lambda)$, for all
$1\leq i\leq n$, we can consider $P_K=\conv\{K\cap\Lambda\}$. If $\dim
P_K<n$ then it suffices to consider \eqref{eq:conj_genP} for $P_K$ and
with respect to the lattice $\Lambda\cap\lin(P_K)$, where $\lin(\cdot)$ denotes the
linear hull.

Furthermore, since any lattice $\Lambda\in\L^n$ can be written as $A\Z^n$ for some invertible matrix $A\in\R^{n\times n}$, and $\lambda_i(K,A\Z^n)=\lambda_i(A^{-1}K,\Z^n)$, we can also restrict to the case $\Lambda=\Z^n$. This reduction to lattice polytopes allows us to utilize Ehrhart theory which is a very active research topic in recent years. Its origin goes back to a work of Eug\`ene Ehrhart \cite{Ehrhart1962} from 1962 who showed that for a given lattice polytope $P\in\P^n$ the function $k\mapsto\LE(kP)$ is a polynomial in $k\in\N$ of degree $n$. Thus,
\begin{equation*}
G(kP)=\sum_{i=0}^n\lE_i(P)k^i,
\end{equation*}
where $\lE_i(P)$ depends only on $P$ and is said to be the $i$th Ehrhart coefficient of $P$. Ehrhart already noticed that $\lE_n(P)=\vol(P),\lE_0(P)=1$ and $\lE_{n-1}(P)$ is the normalized surface area of $P$ (see \cite{BeckRobins2007} for details). Moreover, it can be easily seen that the coefficient $\lE_i$ is homogeneous of degree $i$. Having this by hand, instead of \eqref{eq:conj_genP}, one can consider the somewhat weaker inequality
\begin{equation}\label{eq:conj_weaker}
 \LE(P)\leq\prod_{i=1}^n\left(\frac1{\lambda_i(DP)}+1\right). 
\end{equation}
Let $\LL(P)$ denote the right hand side of this inequality. Then   
\begin{equation*}
\LL(P)=\prod_{i=1}^n\left(\frac{1}{\lambda_i(DP)}+1\right)=\sum_{i=0}^n\sigma_i\left(\frac1{\lambda_1(DP)},\dots,\frac1{\lambda_n(DP)}\right), 
\end{equation*}
where $\sigma_i$ denotes the $i$th elementary symmetric polynomial of
$n$ numbers $x_j$, i.e.,
$\sigma_i(x_1,\dots,x_n)=\sum_{I\subseteq[n],\,\#I=i}\prod_{j\in
  I}x_j$, where $[n]=\{1,\dots,n\}$, and
$\sigma_0(x_1,\dots,x_n)=1$. 
For short we will just write
\begin{equation*}
\sigma_i(P)=\sigma_i\left(\frac{1}{\lambda_1(DP)},\dots,
  \frac{1}{\lambda_n(DP)}\right).
\end{equation*}
 With this notation inequality
\eqref{eq:conj_weaker} is equivalent to $\LE(P)\leq\LL(P)$ and we may
ask whether the coefficient-wise inequalities 
\begin{equation}
\lE_i(P)\leq\sigma_i(P)\label{ineq_coeff_wise}
\end{equation}
 hold for all $i=0,\dots,n$. The case $i=0$ is trivial since in this
 case both sides are equal to 1. For $i\geq 1$  the question is supported by
 two known inequalities in this list. First of all, we have
 $\lE_n(P)\leq\sigma_n(P)$, which follows from Theorem \ref{thm_mink}
 after applying the Brunn-Minkowski inequality (see
 \cite[Thm.~8.1]{Gruber2007}) to derive
 $\lE_n(P)=\vol(P)\leq\frac1{2^n}\vol(DP)$. And secondly, in
 \cite{HenkSchuermannWills2005} it was proved that $\lE_{n-1}(P)\leq\sigma_{n-1}(P)$, for any lattice polytope $P\in\P^n_0$.

Unfortunately, for $i\neq n,n-1$, the inequalities do not hold in general.
\begin{proposition}\label{prop_badexample_coeff} 
Let $Q^n_l=\conv\left\{lC_{n-1}\times\{0\},\pm e_n\right\}$, where
$l\in\N$ and $C_n=[-1,1]^n$ is the cube of edge length 2 centered at
the origin. Then, for
 $n\geq3$ and any constant ${\rm c}$ there exists an $l\in\N$ such
 that  $\lE_{n-2}(Q^n_l)> {\rm c}\,\sigma_{n-2}(Q^n_l)$. If $n\geq4$, we have the same situation for $\lE_{n-3}(Q^n_l)$.
\end{proposition}

The proof of this statement is given at the end of the paper.
In this work we show that for special classes of lattice polytopes,
however,  the coefficient-wise approach leads to positive results.\par

One of these classes is the family of lattice zonotopes. In general, a zonotope $Z$ is the Minkowski sum of finitely many line segments, % or equivalently an affine image of some % possibly high-dimensional 
% cube,
that is, there is a set of vectors $v_1,\dots,v_m\in\R^n$ and points $p_1,\dots,p_m\in\R^n$ such that
\begin{equation*}
Z=\sum_{i=1}^m[p_i,p_i+v_i]=\left\{\sum_{i=1}^m (p_i+\alpha_i\,v_i) : 0\leq \alpha_i \leq 1\right\}.
\end{equation*}
Particularly, zonotopes possess a center of symmetry and furthermore are characterized in the class of centrally symmetric polytopes by the property that all two-dimensional faces are themselves centrally symmetric (see \cite[Thm.~3.3]{Bolker1969}). Zonotopes appear in many different contexts, for instance, in the theory of hyperplane arrangements (cf. \cite[Lect.~7]{Ziegler1995}) and in problems on approximation of convex bodies (cf. \cite[Sect.~15.2]{Gruenbaum2003}).\par
Since we are only interested in lattice zonotopes, i.e.,
$p_i,v_i\in\Z^n$,  and since  \eqref{eq:conj_genP} is invariant under
translations by lattice vectors, we can simply consider lattice
zonotopes  given as the sum of line segments $[0,v_i]$, with
$v_i\in\Z^n$. Our first result shows that for any lattice
parallelepiped $Z$ the coefficient-wise inequalities hold true and, in
particular, we obtain \eqref{eq:conj_genP}. 

\begin{theorem}\label{thm_parall}
Let $Z\in\P^n$ be an $n$-dimensional lattice parallelepiped.  Then
\begin{equation*}
 \lE_i(Z)\leq \sigma_i\left(Z\right),\quad
i=0,\dots,n.
\end{equation*}
\end{theorem}
We note that these inequalities are best possible. For instance, consider
the cube $Z=[0,1]^n=\sum_{i=1}^n [0,e_i]$, where $e_i$ denotes the $i$th standard unit vector. We have
$\lambda_i(DZ)=\lambda_i([-1,1]^n)=1$, and $\LE(k\,Z)=(k+1)^n$ for any integer
$k\in\N$; thus  
$\lE_i(Z)=\binom{n}{i}=\sigma_i(Z)$. 
For general lattice zonotopes $Z$ we obtain a relation up to a factor depending only on the dimension and not on the number of generators.

\begin{theorem}\label{thm_gi_bound}
Let $Z\in\P^n$ be an $n$-dimensional zonotope. Then 
\begin{equation*}
\frac{\lE_i(Z)}{\vol(Z)}\leq\frac{n!}{i!}\prod_{j=i+1}^n\lambda_j(DZ),\quad
i=0,\dots,n.
\end{equation*}
In particular, we get $\lE_i(Z)\leq\frac{n!}{i!}\sigma_i\left(Z\right)$.
\end{theorem}

The second class of polytopes that we consider was introduced by Liu
\cite{Liu2009}, the so called {\em lattice-face polytopes}. In order
to state the definition, let $\pi^{(n-i)}:\R^n\to\R^i$ be the
projection that forgets the last $n-i$ coordinates, $i=1,\dots,n$,
where $\pi^{(0)}$ denotes the identity.    
\begin{definition}[Lattice-face polytopes]\label{def_latface}
A polytope $P\in\P^n$ is called a {\em lattice-face polytope}, if for any $0\leq k\leq n-1$ and any subset $U\subset\mathrm{vert}(P)$ that spans a $k$-dimensional affine space, $\pi^{(n-k)}(\aff(U)\cap\Z^n)=\Z^k$.
\end{definition}

For example, any integral cyclic polytope, i.e., the convex hull of finitely many
lattice points on the moment curve $t\mapsto(t,t^2,\ldots,t^n)$, is
lattice-face (cf.~\cite{BDDPS2005, Liu2009}). In \cite{Liu2009} it is
also shown that lattice-face polytopes are necessarily lattice polytopes and moreover, that every combinatorial type of a rational polytope has a representative among lattice-face polytopes.

\begin{theorem}\label{thm_latface}
Let $P\in\P^n$ be a lattice-face polytope.
\begin{itemize}
 \item[\romannumeral1)] If $P$ is $0$-symmetric, then, for $1\leq i\leq n$,
  \begin{equation*}
    \lE_i(P)\leq \sigma_i\left(P\right).
  \end{equation*}
 \item[\romannumeral2)] If $0\in\mathrm{vert}(P)$ and $SP=\conv(P,-P)$, then, for $1\leq i\leq n$,
  \begin{equation*}
    \lE_i(P)\leq \sigma_i\left(\frac{2}{\lambda_1(SP)},\dots,\frac{2}{\lambda_n(SP)}\right).
  \end{equation*}
\end{itemize}
\end{theorem}

The paper is organized as follows. In Section 2 a geometric
description of the Ehrhart coefficients of lattice zonotopes is
discussed and the proofs of Theorem \ref{thm_parall} and
\ref{thm_gi_bound} are given. Also, some further results on
coefficient-wise inequalities are described, which are obtained by
adding some extra conditions on the generators. In Section 3 we give
a brief introduction to lattice-face polytopes and the proof of
Theorem \ref{thm_latface}. We close the paper with the proof of Proposition \ref{prop_badexample_coeff}.

\section{Lattice zonotopes}

Let $v_1,\dots,v_m\in\Z^n$ and consider $Z=\sum_{i=1}^m[0,v_i]$. 
Concerning the coefficients $\lE_i(Z)$ of the Ehrhart polynomial of $Z$, Betke and Gritzmann  \cite{BetkeGritzmann1986} showed that
\begin{equation}
   \lE_i(Z)=\sum_{F\in\fac_i(Z)} \gamma(F,P)\,\frac{\vol_i(F)}{\det(\aff F\cap\Z^n)},
\label{eq:bg_zonotope}
\end{equation}
where $\fac_i(Z)$ denotes the set of all $i$-faces of $Z$,
$\gamma(F,P)$ the external angle of $F$ at $P$
(cf. \cite[p.~308]{Gruenbaum2003}), 
and $\det(\aff F\cap\Z^n)$ the determinant of the sublattice of $\Z^n$ contained in the affine hull of $F$. Another presentation was given by Stanley \cite[Exer.~31, p.~272]{Stanley1997}
\begin{equation}
   \lE_i(Z)=\sum_{X\in\set_i(Z)} \gcd(i\text{-minors of }X),
\label{eq:stanley_zonotope}
\end{equation}
where $\set_i(Z)$ denotes the set of all linearly independent $i$-element subsets of $\{v_1,\dots,v_m\}$ and $\gcd(a_1,\dots,a_k)$ is the greatest common divisor of the integers $a_1,\dots,a_k\in\Z$. From \eqref{eq:stanley_zonotope} -- as well as \eqref{eq:bg_zonotope} --
we can get a slightly more geometric description: To this end we
denote for a given $J\subseteq[m],\,\#J=i$, the zonotope
generated by the vectors $v_{j}$, $j\in J$, by $P_J$, that is,
$P_J=\sum_{j\in J}[0,v_j]=\left\{\sum_{j\in J} \mu_j\,v_j : 0\leq \mu_j\leq 1\right\}$.

\begin{proposition}\label{prop_g_i}
For $1\leq i\leq n$ we have
\begin{equation*}
 \lE_i(Z)=\sum_{J\subseteq[m],\,\#J=i} \frac{\vol_i(P_J)}{\det(\lin P_J \cap\Z^n)}.
\end{equation*}
\end{proposition}
\begin{proof}
If the vectors $v_{j}$, $j\in J$, are linearly dependent, then $\vol_i(P_J)=0$ and so any non-trivial contribution in that sum comes from an $i$-dimensional parallelepiped. The index of a sublattice $\Lambda'$ of $\Lambda\in\L^n$ is defined as $\ind \Lambda'=\frac{\det \Lambda'}{\det \Lambda}$ (cmp.~\cite[Sect.~1.1]{Martinet2003}). Thus, by the definition of the determinant of a lattice, these non-trivial contributions are just the index of the sublattice generated by $v_{j}$, $j\in J$, with respect to the lattice $\lin P_J\cap\Z^n$.\par
Without loss of generality let $\{v_j: j\in J\}=\{v_1,\dots,v_i\}=V_J$ and let the vectors be linearly independent. First we observe that
\begin{equation}
 V_J \text{ is a lattice basis of }\lin V_J\cap\Z^n \Leftrightarrow \gcd(i\text{-minors of }V_J)=1.
\label{eq:claim1}
\end{equation}
For the ``if-part'' assume that $V_J$ is not a basis of $\lin
V_J\cap\Z^n$ but let $\bar{V}$ be an $n\times i$ matrix
whose columns constitute a basis of the lattice. Then there exists a
matrix $D_J\in\Z^{i\times i}$ with $V_J=\bar{V}\,D_J$ and so $|\det D_J|$ is a divisor of each $i$-minor of $V_J$. Since $|\det D_J|\geq 2$ we get the desired contradiction.
In order to show the ``only if-part'' we extend the vectors in $V_J$ to a basis $\tilde{V}$ of $\Z^n$ of determinant $1$. Developing that determinant with respect to the last $n-i$ columns yields
\begin{equation*}
1=\det\tilde{V}=\sum_{i\text{-minors $\mu_k$ of }V_J} \rho_k\,\mu_k
\end{equation*}
for some integers $\rho_k$. Hence, $\gcd(i\text{-minors of }V_J)=1$.

Next, let $\Lambda_J$ be the lattice generated by $v_1,\dots,v_i$. Then for the index of $\Lambda_J$ with respect to $\lin\,V_J\cap\Z^n$ holds
\begin{equation}
 \ind\,\Lambda_J=\gcd(i\text{-minors of }V_J).
\label{eq:claim2}
\end{equation}
To see this, we use the same notation as in the ''if-part'' above and have $V_J=\bar{V}\,D_J$. Since  $\det D_J=\ind \Lambda_J$ we conclude that $ \ind\,\Lambda_J$ is a divisor of $\gcd(i\text{-minors of }V_J)$. On the other hand we conclude from \eqref{eq:claim1} that $\gcd(i\text{-minors of }\bar{V})=1$ which implies the reverse divisibility. Obviously, \eqref{eq:claim2}, \eqref{eq:stanley_zonotope} and the observation at the beginning of the proof imply the assertion.
\end{proof}

Since $\vol(Z)=\lE_n(Z)$, Proposition \ref{prop_g_i}  is for $i=n$  just the
well-known volume formula of a zonotope $Z=\sum_{i=1}^m[0,w_i]$,
$w_i\in\R^n$, (cf.~\cite{Shephard1974})
\begin{equation}
\vol(Z)=\sum_{1\leq j_1<j_2<\cdots <j_n\leq m} |\det(w_{j_1},\dots,w_{j_n})|.
\label{eq:volumezonotope}
\end{equation}

In order to prove Theorem \ref{thm_parall} we need two auxiliary lemmas.
In the following, for a set $M$ and some $i\in\N$ we denote by $\binom{M}{i}$ the collection of all $i$-element subsets of $M$.
\begin{lemma}\label{lem:bijection}
Let $\{b_1,\dots,b_n\}$ and $\{a_1,\dots,a_n\}$ be two bases of an $n$-dimen\-sional vector space $V$, and let $i\in\{1,\dots,n-1\}$.
Then there exists a bijection $\phi:\binom{[n]}{i}\rightarrow\binom{[n]}{n-i}$ such that
$\{b_k:k\in I\}\cup\{a_j: j\in\phi(I)\}$ is a basis of $V$, for all $I\in\binom{[n]}{i}$.
\end{lemma}

\begin{proof}
We use a standard linear algebra argument involving the exterior
algebra $\Lambda(V)=\oplus_{i=0}^n \Lambda_i(V)$ of $V$ for which we
refer to \cite[Ch.~XVI]{BirkhoffMacLane1979}.
For all $I\in\binom{[n]}{i}$ and $J\in\binom{[n]}{n-i}$ let
$b_I=\wedge_{k\in I}b_k\in\Lambda_i(V)$ and $a_J=\wedge_{j\in J}a_j\in\Lambda_{n-i}(V)$, respectively.
Consider the square matrix $M$ with row index set $\binom{[n]}{i}$ and
column index set $\binom{[n]}{n-i}$, whose $(I,J)$-entry is $b_I\wedge
a_J$. First we note that 
\begin{equation}
 \det M\ne 0.
\end{equation}
Assume the contrary and suppose that some non-trivial linear
combination of the rows of $M$ is zero, say
\[\sum_{I\in\binom{[n]}{i}} c_I\,\left(b_I\wedge a_J\right)=\bigg(\sum_{I\in\binom{[n]}{i}} c_I\,b_I\bigg)\wedge a_J=0,\]
for all $J\in\binom{[n]}{n-i}$, with scalars $c_I$, not all zero.
Expanding the nonzero vector $\sum_{I\in\binom{[n]}{i}} c_I\,b_I\in\Lambda_i(V)$ in terms of the basis \{$a_I : I\in\binom{[n]}{i}\}$ of $\Lambda_i(V)$ yields
\[\bigg(\sum_{I\in\binom{[n]}{i}} d_I\,a_I\bigg)\wedge a_J=\sum_{I\in\binom{[n]}{i}} d_I\,\left(a_I\wedge a_J\right)=0,\]
for all $J\in\binom{[n]}{n-i}$, with scalars  $d_I$, not all zero.
But in view of $a_I\wedge a_J\not=0$ if and only if $I=[n]\setminus J$ we conclude that $d_I=0$, for all $I\in\binom{[n]}{i}$, a contradiction.

So $\det M\ne 0$, and by  Leibniz' formula there exists 
a bijection $\phi:\binom{[n]}{i}\rightarrow\binom{[n]}{n-i}$
with $b_I\wedge a_{\phi(I)}\neq 0$, for $I\in\binom{[n]}{i}$.
This is equivalent to $\{b_k:k\in I\}\cup\{a_j: j\in\phi(I)\}$ being a
basis of $V$, for $I\in\binom{[n]}{i}$
(cf.~\cite[Thm.~XVI.13]{BirkhoffMacLane1979}), which we wanted to show.
\end{proof}

\begin{lemma}\label{lem:ineq_lambdas} 
Let $K\in\K_0^n$, and let $a_1,\dots,a_n\in\Z^n$ be
  linearly independent such that $a_j\in\lambda_j(K)\,K$, $1\leq
  j\leq n$. Let $\overline{L}$ be
  an $i$-dimensional linear subspace, $i\in\{1,\dots,n-1\}$, containing
  $i$ linearly independent points of $\Z^n$, and assume that 
$\lin\{a_{j_1},\dots,a_{j_{n-i}}\}\cap\overline{L}=\{0\}$. Then 
\begin{equation*}
 \prod_{j=1}^i \lambda_j(K\cap\overline{L},\Z^n\cap\overline{L})
\geq \prod_{k\notin\{j_1,\dots,j_{n-i}\}}\lambda_{k}(K).
\end{equation*}
\end{lemma}

\begin{proof} For abbreviation we set $\overline{\Lambda}=\Z^n\cap\overline{L}$, $\overline{K}=K\cap
  \overline{L}$,
  $\overline{\lambda}_j=\lambda_j(K\cap\overline{L},\Z^n\cap\overline{L})$,
  $1\leq j\leq i$, 
  and $\lambda_j=\lambda_j(K)$, $1\leq j\leq n$. Moreover, let
  $\overline{w}_1,\dots,\overline{w}_i\in\overline{\Lambda}$ be
  linearly independent such that $\overline{w}_j\in\overline{\lambda}_j\,\overline{K}$.
  Let $j_1<j_2<\dots <j_{n-i}$ and let 
  $k_1<k_2<\dots<k_{i}$ be the indices in
  $[n]\setminus\{j_1,\dots,j_{n-i}\}$. Suppose there exists
  an index $l\in\{1,\dots,i\}$ with  
\begin{equation}
           \overline{\lambda}_l< \lambda_{k_l},
\label{eq:equation1}
\end{equation}
 and let $m$ be the smallest index such that
$\lambda_m=\lambda_{k_l}$.  Since $\overline{K}\subset K$, 
$\overline{\Lambda}\subset\Z^n$, we get by \eqref{eq:equation1}, the choice of $m$ and the
definition of the successive minima that   
\begin{equation*}
     \{\overline{w}_1,\dots,\overline{w}_l\}\cup\{a_j : 1\leq j\leq
     m-1, j\in\{j_1,\dots,j_{n-i}\}\}\subseteq\inter(\lambda_m\, K)\cap\Z^n.
\end{equation*}  
Since there are at most $l-1$ indices in the set $\{1,\dots,m-1\}$
belonging to $\{k_1,\dots,k_i\}$, we conclude that 
$\#\{j : j\in\{j_1,\dots,j_{n-i}\}\text{ and }1\leq j\leq m-1\}\geq m-l$.
 Hence, on the left
hand side of the inclusion above  we have at least $m$ lattice vectors 
which by the assumption $\lin\{a_{j_1},\dots,a_{j_{n-i}}\}\cap\overline{L}=\{0\}$ are  linearly
independent. This, however,  contradicts the definition of
$\lambda_m$, and so we have
 shown $\overline{\lambda}_l\geq \lambda_{k_l}$, $l=1,\dots,i$, which implies the
assertion.
\end{proof}

\begin{proof}[Proof of Theorem \ref{thm_parall}] 
Let $Z$ be the parallelepiped generated by $v_1,\dots,v_n\in\Z^n$. Abbreviate $\lambda_j(DZ)$ by $\lambda_j$ and for
$J\subseteq[n]$ with $\#J=i$, let $DP_J=P_J-P_J=\left\{\sum_{j\in J}
  \mu_j\,v_j : -1\leq \mu_j\leq 1\right\}$ and write
$\Lambda_J=\lin\{v_{j}:j\in J\}\cap\Z^n$. In view of Proposition
\ref{prop_g_i} and the fact that
$\vol_i(P_J)=\frac1{2^i}\vol_i(DP_J)$ we have to show 
\begin{equation*}
 \lE_i(Z)=\frac1{2^i}\sum_{J\subseteq[n],\,\#J=i} \frac{\vol_i(DP_J)}{\det\Lambda_J} 
 \leq 
 \sum_{I\subseteq[n],\,\#I=i} \frac1{\prod_{k\in I} \lambda_k}.
\end{equation*}
By the second theorem of Minkowski 
(Theorem \ref{thm_mink}) we can estimate each summand on the left and get
\begin{equation*}
 \lE_i(Z)=\frac1{2^i}\sum_{J\subseteq[n],\,\#J=i} \frac{\vol_i(DP_J)}{\det\Lambda_J} 
 \leq \sum_{J\subseteq[n],\,\#J=i} \frac1{\prod_{j=1}^i \lambda_j(DP_J,\Lambda_J)}.
\end{equation*}
Hence it suffices to show
\begin{equation}
\sum_{J\subseteq[n],\,\#J=i} \frac1{\prod_{j=1}^i
  \lambda_j(DP_J,\Lambda_J)}
\leq 
 \sum_{I\subseteq[n],\,\#I=i} \frac1{\prod_{k\in I} \lambda_k}.
\label{eq:first_step}
\end{equation}
Now, let $a_1,\ldots,a_n\in\Z^n$  be linearly independent with $a_j\in\lambda_j\,DZ$, $1\leq j\leq n$. Furthermore
$v_1,\ldots,v_n\in\Z^n$ are linearly independent as well. Thus by Lemma \ref{lem:bijection} there is a bijection
$\phi:\binom{[n]}{i}\rightarrow\binom{[n]}{n-i}$ such that for all $J\in\binom{[n]}{i}$
\[\lin\{v_j:j\in J\}\cap\lin\{a_k: k\in\phi(J)\}=\{0\}.\]
Thus together with Lemma \ref{lem:ineq_lambdas} we get
\begin{equation*}
 \prod_{j=1}^i \lambda_j(DZ\cap\lin\{v_l:l\in J\},\Z^n\cap\lin\{v_l:l\in J\})
\geq \prod_{k\notin\phi(J)}\lambda_{k},
\end{equation*}
and on account of $\lambda_j(DP_J,\Lambda_J)\geq
\lambda_j(DZ\cap\lin\{v_l:l\in J\},\Z^n\cap\lin\{v_l:l\in J\})$ we
obtain   
\begin{equation}
 \frac1{\prod_{j=1}^i \lambda_j(DP_J,\Lambda_J)}
\leq \frac1{\prod_{k\notin\phi(J)}\lambda_{k}}.
\label{eq:bijecti}
\end{equation} 
Since $\phi$ is a bijection we get \eqref{eq:first_step}.
\end{proof}

For the proof of Theorem \ref{thm_gi_bound} we need the following counterpart to 
Min\-kowski's Theorem \ref{thm_mink} (e.g.~see \cite[Thm.~1.2]{HenkWills2008})
\begin{equation}\label{eq:mink_lowbound}
\frac{2^n}{n!}\det(\Lambda)\leq\lambda_1(K,\Lambda)\cdot\ldots\cdot\lambda_n(K,\Lambda)\vol(K),
\end{equation}
where $K\in\K^n_0$ and $\Lambda\in\L^n$. 

\begin{proof}[Proof of Theorem \ref{thm_gi_bound}] Let $Z$ be 
  generated by $v_1,\dots,v_m\in\Z^n$ and let $\dim Z=n$.
For short we write $\lambda_i$ instead of $\lambda_i(DZ)$ and for $I\subseteq[m]$, 
$\#I=i$, let $P_I=\left\{\sum_{j\in I} \mu_j\,v_j : 0\leq \mu_j\leq 1\right\}$,
$L_I=\lin\{v_j:j\in I\}$ and $L_I^\bot$ be its orthogonal
complement. The orthogonal projection of a set $S\subseteq \R^n$ onto a linear
subspace $L$ is denoted by $S|L$.

For $J\subseteq[m]$, $\#J=n$, and $i\in[n]$, let $I\subseteq J$ with $\#I=i$. Then
\begin{equation*}
\vol(P_J)=\vol_i(P_I)\,\vol_{n-i}(P_J|L_I^\bot),
\end{equation*}
which, e.g.,  can easily be seen by Gram-Schmidt
orthogonalization.
Hence, by Proposition \ref{prop_g_i} or \eqref{eq:volumezonotope} we
can write
\[\begin{aligned}
\vol(Z)&=\sum_{J\subseteq[m],\,\#J=n}\vol(P_J)\\
&=\sum_{J\subseteq[m],\,\#J=n}\frac1{\binom{n}{i}}\sum_{I\subseteq J,\,\#I=i}\vol_i(P_I)\,\vol_{n-i}(P_J|L_I^\bot)\\
&=\frac1{\binom{n}{i}}\sum_{I\subseteq[m],\,\#I=i}\vol_i(P_I)\sum_{I\subseteq J\subseteq[m],\,\#J=n}\vol_{n-i}(P_J|L_I^\bot).
\end{aligned}\]
Furthermore, for $I\subseteq[m]$ with $\#I=i$, we have
\[\sum_{I\subseteq J\subseteq[m],\,\#J=n}\vol_{n-i}(P_J|L_I^\bot)=\vol_{n-i}(Z|L_I^\bot),\]
because the sum on the left hand side covers all volumes of
$(n-i)$-dimensional parallelepipeds that are spanned by generators of
$Z|L_I^\bot$ (cf.~\eqref{eq:volumezonotope}). This implies
\[\begin{aligned}
\vol(Z)&=\frac1{\binom{n}{i}}\sum_{I\subseteq[m],\,\#I=i}\vol_i(P_I)\vol_{n-i}(Z|L_I^\bot)\\
&=\frac1{\binom{n}{i}}\sum_{I\subseteq[m],\,\#I=i}\frac{\vol_i(P_I)}{\det(\Z^n\cap L_I)}\,\frac{\vol_{n-i}(Z|L_I^\bot)}{\det(\Z^n|L_I^\bot)},
\end{aligned}\]
where for the last step we refer to \cite[Corollary 1.3.5]{Martinet2003}.
Together with the identity $\vol_{n-i}(Z|L_I^\bot)=\frac1{2^{n-i}}\vol_{n-i}(DZ|L_I^\bot)$ and \eqref{eq:mink_lowbound} we get
\[\vol(Z)\geq\frac1{\binom{n}{i}}\sum_{I\subseteq[m],\,\#I=i}\frac{\vol_i(P_I)}{\det(\Z^n\cap L_I)}
\left(\frac1{(n-i)!}\prod_{j=1}^{n-i}\frac1{\lambda_j(DZ|L_I^\bot,\Z^n|L_I^\bot)}\right).\]
Since $\lambda_j(DZ|L_I^\bot,\Z^n|L_I^\bot)\leq\lambda_{i+j}(DZ)$, for $j=1,\ldots,n-i$, we obtain
\begin{equation*}
\vol(Z)\geq\frac{i!}{n!}\sum_{I\subseteq[m],\,\#I=i}\frac{\vol_i(P_I)}{\det(\Z^n\cap L_I)}\prod_{j=i+1}^{n}\frac1{\lambda_j}.
\end{equation*}
With Proposition \ref{prop_g_i} we finally obtain
\begin{equation}
\vol(Z)\geq\frac{i!}{n!}\lE_i(Z)\prod_{j=i+1}^{n}\frac1{\lambda_j},
\end{equation}
as desired. The second part of the theorem can now be derived with the help of $\vol(DZ)=2^n\vol(Z)$ and Theorem \ref{thm_mink}.
\end{proof}

We remark that Henk, Linke and Wills \cite[Cor.~1.1]{HenkLinkeWills2010} improved 
the bound \eqref{eq:mink_lowbound} for the class of zonotopes by, roughly speaking, 
a factor of order $(\sqrt{n})^{n+1}$, which leads to the better inequalities
\begin{equation*}
\lE_i(Z)\leq\binom{n}{i}(n-i)^{\frac{n-i}2}\sigma_i\left(Z\right),\,\textrm{for }1\leq i\leq n.
\end{equation*}

The remaining part of this section will be devoted to some partial results 
concerning the coefficient-wise approach to Conjecture \ref{conj_main} in the
case when one imposes additional assumptions on the generators of a
lattice zonotope.\par
The first one is an extension of Theorem \ref{thm_parall} and 
depending on the number of generators it improves upon Theorem
\ref{thm_gi_bound}.

\begin{theorem} Let $\{v_1,\dots,v_m\}\subset\Z^n$ be in general
  position, i.e., every $n$ of them are linearly independent, and let $Z\in\P^n$ 
  be the zonotope generated by these vectors. 
  Then, for $1\leq i\leq n$,
  \begin{equation*}
   \lE_i(Z)\leq \frac{\binom{m}{i}}{\binom{n}{i}}\sigma_i\left(Z\right).
  \end{equation*}
\end{theorem}
\begin{proof}
We follow the outline of the proof of Theorem \ref{thm_parall} and
also use its notation.  Based on Proposition \ref{prop_g_i} and
Minkowski's second theorem (Theorem \ref{thm_mink}) 
here it suffices to show (cf.~\eqref{eq:first_step}) 
\begin{equation}
\sum_{J\subseteq[m],\,\#J=i} \frac1{\prod_{j=1}^i
  \lambda_j(DP_J,\Lambda_J)}
\leq
    \frac{\binom{m}{i}}{\binom{n}{i}} \sum_{I\subseteq[n],\,\#I=i} \frac1{\prod_{k\in I} \lambda_k}.
\label{eq:first_step_gp}
\end{equation}
Now, since every set $J\subseteq[m]$ with $\#J=i$ is contained in 
$\binom{m-i}{n-i}$ sets $I\subseteq[m]$ of size 
$\#I=n$, we can replace the left hand side by 
\begin{equation*}
  \frac1{\binom{m-i}{n-i}}\sum_{I\subseteq[m],\,\#I=n}
 \sum_{J\subseteq I,\,\#J=i} \frac1{\prod_{j=1}^i \lambda_j(DP_J,\Lambda_J)}.
\end{equation*}
and \eqref{eq:first_step_gp} becomes 
\begin{equation}
 \sum_{I\subseteq[m],\,\#I=n}
 \sum_{J\subseteq I,\,\#J=i} \frac1{\prod_{j=1}^i
   \lambda_j(DP_J,\Lambda_J)}\leq 
\binom{m}{n} \sum_{I\subseteq[n],\,\#I=i} \frac1{\prod_{k\in I} \lambda_k}.
 \label{eq:second_step_gp}
\end{equation}
Now let
$a_1,\ldots,a_n\in\Z^n$  be linearly independent with
$a_j\in\lambda_j\,DZ$, $1\leq j\leq n$. By our assumption, 
any choice of $n$ generators $v_{i_1},\ldots,v_{i_n}\in\Z^n$ is linearly
independent and so we may apply Lemma  \ref{lem:bijection} to
any $n$-subset $I=\{i_1,\dots,i_n\}\subseteq[m]$. Hence, as in the
proof of Theorem \ref{thm_parall} we find that  there is a bijection
$\phi:\binom{I}{i}\rightarrow\binom{[n]}{n-i}$ such that for all
$J\in\binom{I}{i}$ (cf.~\eqref{eq:bijecti})
\begin{equation*}
 \frac1{\prod_{j=1}^i \lambda_j(DP_J,\Lambda_J)}
\leq \frac1{\prod_{k\notin\phi(J)}\lambda_{k}}.
\end{equation*} 
Since $\phi$ is a bijection we get 
\begin{equation*}
 \sum_{J\subseteq I,\,\#J=i} \frac1{\prod_{j=1}^i \lambda_j(DP_J,\Lambda_J)}
 % \leq \sum_{J\subseteq I,\,\#J=i}\frac1{\prod_{k\notin\phi(J)}\lambda_{k}}
 \leq \sum_{T\subseteq[n],\,\#T=i}\frac1{\prod_{t\in T}\lambda_{t}},
\end{equation*}
which implies \eqref{eq:second_step_gp}.
\end{proof}

As an immediate consequence of Theorem \ref{thm_mink} one can prove (\ref{ineq_coeff_wise}) for $i=1$ and lattice zonotopes with primitive generators in general position. Here a non-trivial lattice vector $z\in\Z^n$ is said to be {\em primitive}, if the greatest common divisor of its entries equals one.

\begin{corollary}\label{cor_primgenpos}
Let $\{v_1,\dots,v_m\}\subset\Z^n$ be primitive vectors  in general
  position, and let $Z\in\P^n$ be the zonotope generated by these vectors.
Then
\begin{equation*}
 \lE_1(Z) = m \leq \sum_{i=1}^n \frac1{\lambda_i(DZ)}=\sigma_1\left(Z\right).
\end{equation*}
\end{corollary}
\begin{proof} First, by \eqref{eq:stanley_zonotope} it holds $\lE_1(Z)=\sum_{i=1}^m\gcd(v_i)$, which equals $m$ because the $v_i$ are chosen to be primitive. Moreover, the generators are in general position and any parallelepiped with integer vertices has volume at least one, which yields -- using also Proposition \ref{prop_g_i} -- that $\vol(Z)=\lE_n(Z)\geq \binom{m}{n}$ and together with $\vol(Z)=\frac1{2^n}\vol(DZ)$ and Theorem \ref{thm_mink} we conclude that
\begin{eqnarray*}
  2^n &\geq&\lambda_1(DZ)\cdot\ldots\cdot\lambda_n(DZ)\,\vol(DZ)\\
      &=&2^n\lambda_1(DZ)\cdot\ldots\cdot\lambda_n(DZ)\,\vol(Z)\geq 2^n \lambda_1(DZ)\cdot\ldots\cdot\lambda_n(DZ)\,\binom{m}{n}.
\end{eqnarray*}
Thus,
\begin{equation*}
 \frac{1}{\lambda_1(DZ)}\cdot\ldots\cdot\frac{1}{\lambda_n(DZ)} \geq \binom{m}{n}
\end{equation*}
and the inequality of the arithmetic and geometric mean finally yields
\begin{equation*}
 \frac{1}{\lambda_1(DZ)}+\cdots + \frac{1}{\lambda_n(DZ)} \geq n\binom{m}{n}^{1/n}\geq m.
\end{equation*}
\end{proof}

In the context of $\lE_1(Z)$ it might be also of interest to have a look at
the so called {\em Davenport constant} $s(G)$ of a finite Abelian group $G$: it is the minimal $d$ such that every sequence of $d$ elements of $G$ contains a nonempty subsequence with zero-sum. For a survey on this and related zero-sum problems see \cite{GaoGeroldinger2006} and the references therein.  It is conjectured that
\begin{equation*}
    s(\Z^n_k)=n(k-1)+1,
\end{equation*}
where $\Z^n_k$ is the $n$-fold product of the cyclic group $\Z_k$ of
order $k$.
The conjecture is known to be true if  $k$ is a prime power
(cf. \cite{Olsen1969}), and so  we get, for instance,

\begin{proposition} Let $k\in\N$ be a prime power, and let $m\in\N$
  such that $n(k-1)+1\leq m\leq kn$. Let $Z\in\P^n$ be a zonotope
  generated by $m$ primitive lattice vectors. Then
\begin{equation*}
 \lE_1(Z)\leq  n\,\frac1{\lambda_1(DZ)}.
\end{equation*}
\end{proposition}
\begin{proof}
As in the proof of Corollary \ref{cor_primgenpos} we have $\lE_1(Z)=m$ and so we have to show that $\lambda_1(DZ)\leq\frac{n}{m}$.
Let $H=\{x\in\R^n : a^\intercal x=0\}$ be a hyperplane such that the half-space $\{x\in\R^n : a^\intercal x>0\}$ contains, without loss of generality, all the vectors $v_1,\dots,v_m$ (if not replace $v_i$ by $-v_i$, which does not change $DZ$). This implies, that any sum of the generators is non-zero. Since $s(\Z^n_k)=n(k-1)+1\leq m$, there exists a subset $v_{i_1},\dots,v_{i_l}$ of the generators whose sum is divisible by $k$ and so $\lambda_1(DZ)\leq \frac1{k}\leq\frac{n}{m}$ as desired.
\end{proof}

\section{Lattice-face polytopes}

In this section, we study Conjecture \ref{conj_main}\, on the class of
lattice-face polytopes which were already defined in the introduction
(see Definition \ref{def_latface}). First of all, we state some
properties of these polytopes being relevant for our further
discussion. Recall that $\pi^{(n-i)}$ denotes the projection that
forgets the last $n-i$ coordinates, $i=1,\dots,n$. For sake of brevity we write $\pi=\pi^{(1)}$.

\begin{lemma}[cf. \cite{Liu2009}]\label{lem_prop_latfac}
Let $P\in\P^n$ be a lattice-face polytope. Then,
\begin{itemize}
 \item[\romannumeral1)] $\pi(P)\in\P^{n-1}$ is a lattice-face polytope.
 \item[\romannumeral2)] $mP$ is a lattice-face polytope, for any integer $m$.
 \item[\romannumeral3)] Let $H$ be an $(n-1)$-dimensional affine space
   spanned by some subset of $\mathrm{vert}(P)$. Then, for any lattice
   point $y\in\Z^{n-1}$, the preimage $\pi^{-1}(y)\cap H$ is also a
   lattice point.
 \item[\romannumeral4)] $P$ is a lattice polytope.
\end{itemize}
\end{lemma}

As Liu \cite[Thm.~1.1]{Liu2009} showed, the coefficients of the Ehrhart polynomial of lattice-face polytopes have a nice geometric meaning.

\begin{theorem}[Liu, 2009]\label{thm_liu}
Let $P\in\P^n$ be a lattice-face polytope. Then
 \begin{equation*}
\LE(P,k)=\sum_{i=0}^n\vol_i(\pi^{(n-i)}(P))k^i,
\end{equation*}
where $\vol_0(\pi^{(n)}(P)):=1$.
\end{theorem}

This will be our starting point to prove Theorem \ref{thm_latface}\,. But first, we need an auxiliary lemma that relates the successive minima of lattice-face polytopes to those of their projections.

\begin{lemma}\label{lem_latface}
Let $P\in\P^n$ be a lattice-face polytope.
\begin{itemize}
 \item[\romannumeral1)] If $P$ is $0$-symmetric, then, for $1\leq
   j\leq i\leq n$,
  \begin{equation*}
    \lambda_j(\pi^{(n-i)}(P),\Z^i)\geq\lambda_j(P).
  \end{equation*}
 \item[\romannumeral2)] If $0\in\mathrm{vert}(P)$ and
   $SP=\conv(P,-P)$, then, for $1\leq j\leq i\leq n$,
  \begin{equation*}
    \lambda_j(\pi^{(n-i)}(SP),\Z^i)\geq\lambda_j(SP).
  \end{equation*}
\end{itemize}
\end{lemma}
\begin{proof}
\romannumeral1): It suffices to show that
$\lambda_j:=\lambda_j(\pi(P),\Z^{n-1})\geq\lambda_j(P)$, for all $j=1,\dots,n-1$. To this end, let $\{z_1,\dots,z_j\}\subset\Z^{n-1}$ be linearly independent lattice points in $\lambda_j\, \pi(P)$. Our first observation is that any set of vectors $\{\bar z_1,\dots,\bar z_j\}\subset\R^n$ with $z_i=\pi(\bar z_i), i=1,\dots,j$, is also linearly independent, because any linear dependence would be preserved by the projection $\pi$. Therefore, we need to show that, for all $i=1,\dots,j$, there is always a lattice point $\bar z_i\in\lambda_j\, P$ such that $z_i=\pi(\bar z_i)$.\par
In order to see this, we fix an $i$ and  set   $z=z_i$ and
$\mu=\lambda_i>0$. In particular, we have $z\in\mu \pi(P)\cap\Z^{n-1}$. Since, $0\in\mu\pi(P)$, there are linearly independent $v_1,\dots,v_{n-1}\in\mathrm{vert}(\pi(P))$ and $\gamma_1,\dots,\gamma_{n-1}\in[0,1]$ with $\sum_{i=1}^{n-1}\gamma_i\leq1$, such that $z=\mu\sum_{i=1}^{n-1}\gamma_i v_i$. For any $v_i$ there is a vertex $\bar v_i$ of $P$ in the preimage of $v_i$ under $\pi$, and these $\bar v_1,\dots,\bar v_{n-1}$ are linearly independent. This means, that the hyperplane $H=\aff\{0,\bar v_1,\dots,\bar v_{n-1}\}=\aff\{\pm\bar v_1,\dots,\pm\bar v_{n-1}\}$ is $(n-1)$-dimensional and spanned by vertices of $P$, because $P=-P$. Therefore, since $P$ is a lattice-face polytope we have by Lemma \ref{lem_prop_latfac} \romannumeral3) that the point $\bar z=\pi^{-1}(z)\cap H$ has integral coordinates. It remains to show that $\bar z$ lies in $\mu P$. The containment of $\bar z$ in $H$ gives us $\beta_1,\dots,\beta_{n-1}\in\R$ such that $\bar z=\sum_{i=1}^{n-1}\beta_i\bar v_i$. Furthermore, it is $$\mu\sum_{i=1}^{n-1}\gamma_i v_i=z=\pi(\bar z)=\sum_{i=1}^{n-1}\beta_i\pi(\bar v_i)=\sum_{i=1}^{n-1}\beta_i v_i,$$ which yields $\beta_i=\mu\gamma_i$, for all $i=1,\dots,n-1,$ because the $v_i$'s were chosen to be linearly independent. So, with $\sum_{i=1}^{n-1}\gamma_i\leq1$, we get $\bar z=\mu\sum_{i=1}^{n-1}\gamma_i\bar v_i\in\mu P$ as claimed.\par
In conclusion, we found the point $\bar z\in\mu P\cap\Z^n$ for which
$z=\pi(\bar z)$ and we are done.

The proof of \romannumeral2) follows the same lines as above. We only
note, that $\mathrm{vert}(SP)\subseteq\{\pm v :  v\in\mathrm{vert}(P)\}$ and the assumption $0\in\mathrm{vert}(P)$ is used to simultaneously control the signs of the vertices which span $H$.
\end{proof}

\begin{remark}
The above lemma does not hold for general polytopes. For example, consider $P_t=\conv\left\{\pm\binom{t-1}1,\pm\binom{t}1\right\}$, $t\in\N$. We have $\lambda_1(P_t,\Z^2)=1$ and $\lambda_1(P_t|e_2^\perp,\Z)=\frac1{t}$. Therefore, there does not even exist a constant depending on the dimension such that the successive minima of the projection could be bounded from below, up to this constant, by those of the original polytope.
\end{remark}

\begin{proof}[Proof of Theorem \ref{thm_latface}]
\romannumeral1): By Theorems \ref{thm_liu} and \ref{thm_mink} we obtain, for all $i=1,\dots,n$, \begin{equation*}
\lE_i(P)=\vol_i(\pi^{(n-i)}(P))\leq\prod_{j=1}^i\frac2{\lambda_j(\pi^{(n-i)}(P),\Z^i)}.
\end{equation*}
Using Lemma \ref{lem_latface} \romannumeral1), we continue this inequality to get
\begin{equation*}
\lE_i(P)\leq\prod_{j=1}^i\frac2{\lambda_j(P)}\leq\sigma_i\left(P\right).
\end{equation*}

Note, that for $i\neq n$ the last inequality sign is actually a strict one.\par
\romannumeral2): By definition it is $P\subset SP$ and so $\vol_i(\pi^{(n-i)}(P))\leq\vol_i(\pi^{(n-i)}(SP))$. Thus, using Lemma \ref{lem_latface} \romannumeral2) we can argue in the same way as in the first part.
\end{proof}

\section{Proof of Proposition \ref{prop_badexample_coeff}}

Recall $Q^n_l=\conv\left\{lC_{n-1}\times\{0\},\pm e_n\right\}$ as the polytope under consideration. By cutting $kQ^n_l$ into lattice slices orthogonal to $e_n$, we find that the Ehrhart polynomial of $Q^n_l$ is given by
\begin{eqnarray*}
\LE(kQ^n_l)&=&(2kl+1)^{n-1}+2\sum_{j=0}^{k-1}(2jl+1)^{n-1}\\
&=&(2kl+1)^{n-1}+2\sum_{j=0}^{k-1}\sum_{i=0}^{n-1}\binom{n-1}{i}(2jl)^i\\
&=&\sum_{i=0}^{n-1}\binom{n-1}{i}(2l)^ik^i+2\sum_{i=0}^{n-1}\binom{n-1}{i}(2l)^i\left(\sum_{j=0}^{k-1}j^i\right).
\end{eqnarray*}
Faulhaber's formula (see \cite[p.~106]{ConwayGuy1996}) expresses the sum $\sum_{j=0}^{k-1}j^i$ as a polynomial in $k$. Plugging this into the above identity and collecting for powers of $k$ yields
\begin{equation*}
\lE_i(Q^n_l)=2(2l)^{i-1}\left(\binom{n-1}{i}l+\sum_{j=i-1}^{n-1}P(i,j)\binom{n-1}{j}(2l)^{j-i+1}\right),
\end{equation*}
where $P(i,j)=\sum_{t=i}^{j+1}\frac{(-1)^{t-i}\binom{j+1}{t}\binom{t}{i}}{j+1}B_{j+1-t}$ and $B_m$ are the Bernoulli numbers, with $B_1=\frac12$ (see \cite[p.~107]{ConwayGuy1996}). Therefore, via $P(n,n-1)=\frac1{n}, P(n,n)=-\frac12, P(n-1,n)=\frac{n}{12}$ and $P(n-2,n)=0$, we obtain
\begin{eqnarray*}
\lE_{n-2}(Q^n_l)&=&(n-1)(2l)^{n-3}\left(\frac23 l^2+1\right)\quad\textrm{and}\\
\lE_{n-3}(Q^n_l)&=&\frac23\binom{n-1}2(2l)^{n-4}\left(2l^2+1\right).
\end{eqnarray*}
The successive minima are $\lambda_1(Q^n_l)=\ldots=\lambda_{n-1}(Q^n_l)=\frac1{l}$ and $\lambda_n(Q^n_l)=1$, from which we get
\begin{equation*}
\sigma_i(Q^n_l)=\binom{n-1}{i}(2l)^i+2\binom{n-1}{i-1}(2l)^{i-1},\textrm{ for }1\leq i\leq n-1.
\end{equation*}
Seen as polynomials in $l$, the $\sigma_i(Q^n_l)$ have degree $i$, whereas $\lE_{n-2}(Q^n_l)$ and $\lE_{n-3}(Q^n_l)$ have degree $n-1$ and $n-2$, respectively. Thus, for $i\in\{n-2,n-3\}$ and any fixed constant ${\rm c}$, there exists an $l\in\N$ such that $\lE_i(Q^n_l)>{\rm c}\,\sigma_i(Q^n_l)$.\par
Note, that Conjecture \ref{conj_main} nevertheless holds for all the polytopes $Q^n_l$. As a final remark, we consider the special case $n=3$. Here, we get
\begin{equation*}
 \LE(kQ^3_l)=\frac83 l^2 k^3 + 4lk^2 + \left(\frac43 l^2+2\right)k + 1,
\end{equation*}
i.e., all Ehrhart coefficients of $Q^3_l$ are positive, and
\begin{equation*}
 \LL(kQ^3_l)=8l^2 k^3 + (4l^2+8l)k^2 + (4l+2)k + 1.
\end{equation*}

\smallskip
\noindent{\it Acknowledgment.} The authors would like to thank the
referees for their very valuable comments, suggestions and
corrections.  We also  would like to thank Mar\'{\i}a
Hern{\'a}ndez Cifre for her help on an earlier draft. 

%---------------------------------bibliography-----------------------------
%\bibliographystyle{amsplain}
%\bibliography{bibfile}

\begin{thebibliography}{10}

\bibitem{BDDPS2005}
{M.~Beck, J.~De Loera, M.~Develin, J.~Pfeifle and R.P.~Stanley},
  \emph{Coefficients and roots of {E}hrhart polynomials}, Contemp. Math.
  \textbf{374} (2005), 15--36.

\bibitem{BeckRobins2007}
M.~Beck and S.~Robins, \emph{Computing the continuous discretely},
  Undergraduate Texts in Mathematics, Springer, New York, 2007, Integer-point
  enumeration in polyhedra.

\bibitem{BetkeGritzmann1986}
U.~Betke and P.~Gritzmann, \emph{An application of valuation theory to two
  problems in discrete geometry}, Discrete Math. \textbf{58} (1986), no.~1,
  81--85.

\bibitem{BetkeHenkWills1993}
{U.~Betke, M.~Henk and J.M.~Wills}, \emph{Successive-minima-type
  inequalities}, Discrete Comput. Geom. \textbf{9} (1993), no.~2, 165--175.

\bibitem{BirkhoffMacLane1979}
G.~Birkhoff and S.~MacLane, \emph{Algebra}, second ed., Macmillan, New York, 1979.

\bibitem{Bolker1969}
E.D.~Bolker, \emph{A {C}lass of {C}onvex {B}odies}, Trans. Amer. {M}ath. {S}oc. \textbf{145} (1969), 323--345.

\bibitem{ConwayGuy1996}
J.H.~Conway and R.K.~Guy, \emph{The book of numbers},
  Springer, New York, 1996.

\bibitem{Ehrhart1962}
E.~Ehrhart, \emph{Sur les poly{\`e}dres rationnels homoth{\'e}tiques {\`a} n
  dimensions}, C. R. Acad. Sci. \textbf{254} (1962), 616--618.

\bibitem{GaoGeroldinger2006}
W.~Gao and A.~Geroldinger, \emph{Zero-sum problems in finite abelian groups: a
  survey}, Expo. Math. \textbf{24} (2006), no.~4, 337--369.

\bibitem{Gruber2007}
P.M.~Gruber, \emph{Convex and {D}iscrete {G}eometry}, Springer, 2007.

\bibitem{GruberLekkerkerker1987}
P.M.~Gruber and C.G.~Lekkerkerker, \emph{Geometry of numbers}, second ed.,
  vol.~37, North-Holland Publishing Co., Amsterdam, 1987.

\bibitem{Gruenbaum2003}
B.~Gr{\"u}nbaum, \emph{Convex polytopes}, 2nd ed., Springer, 2003, Second
  edition prepared by V. Kaibel, V. Klee and G. M. Ziegler.

\bibitem{HenkLinkeWills2010}
{M.~Henk, E.~Linke and J.M.~Wills}, \emph{Minimal zonotopes containing the
  crosspolytope}, Lin. Alg. Appl. \textbf{432} (2010), no.~11, 2942--2952.

\bibitem{HenkSchuermannWills2005}
{M.~Henk, A.~Sch{\"u}rmann and J.M.~Wills}, \emph{Ehrhart polynomials and
  successive minima}, Mathematika \textbf{52} (2005), 1--16.

\bibitem{HenkWills2008}
M.~Henk and J.M.~Wills, \emph{Minkowski's successive minima}, Number theory \&
  discrete geometry, Ramanujan Math. Soc. Lect. Notes Ser., vol.~6, Ramanujan
  Math. Soc., 2008, pp.~129--142.

\bibitem{Liu2009}
F.~Liu, \emph{A note on lattice-face polytopes and their {E}hrhart
  polynomials}, Proc. Amer. Math. Soc. \textbf{137} (2009), no.~10, 3247--3258.

\bibitem{Malikiosis2010}
R.~Malikiosis, \emph{A discrete analogue for {M}inkowski's second theorem on
  successive minima}, \url{http://arxiv.org/abs/1001.3729}.

\bibitem{Malikiosis2009}
\bysame, \emph{An {O}ptimization {P}roblem {R}elated to {M}inkowski's
  {S}uccessive {M}inima}, Discrete Comput. Geom. \textbf{43} (2010), no.~4,
  784--797.

\bibitem{Martinet2003}
J.~Martinet, \emph{Perfect lattices in {E}uclidean
  spaces}, Springer, 2003.

\bibitem{Minkowski1910}
H.~Minkowski, \emph{Geometrie der {Z}ahlen}, Teubner, 1910.

\bibitem{Olsen1969}
J.E.~Olson, \emph{A combinatorial problem on finite {A}belian groups. {I}}, J.
  Number Theory \textbf{1} (1969), 8--10.

\bibitem{Shephard1974}
G.C.~Shephard, \emph{Combinatorial {P}roperties of associated
  {Z}onotopes}, Can. J. Math. \textbf{26} (1974), 302--321.

\bibitem{Stanley1997}
R.~P.~Stanley, \emph{Enumerative combinatorics. {V}ol. 1}, Cambridge Studies in
  Advanced Mathematics, vol.~49, Cambridge University Press, Cambridge, 1997.

\bibitem{Ziegler1995}
G.M.~Ziegler, \emph{Lectures on polytopes}, Springer, 1995.

\end{thebibliography}
\providecommand{\bysame}{\leavevmode\hbox to3em{\hrulefill}\thinspace}
\providecommand{\MR}{\relax\ifhmode\unskip\space\fi MR }
% \MRhref is called by the amsart/book/proc definition of \MR.
\providecommand{\MRhref}[2]{%
  \href{http://www.ams.org/mathscinet-getitem?mr=#1}{#2}
}
\providecommand{\href}[2]{#2}

\end{document}